\documentclass[a4paper,leqno,10pt]{amsart}

\usepackage{epsf,graphicx,epsfig}
\usepackage{amscd}
\usepackage{amsmath,latexsym,amssymb,amsthm}
\usepackage[nospace,noadjust]{cite}
\usepackage{textcomp}
\usepackage{setspace,cite}
\usepackage{lscape,fancyhdr,fancybox}
\usepackage{stmaryrd}
\usepackage[all,cmtip]{xy}
\usepackage{tikz}
\usepackage{cancel}
\usetikzlibrary{shapes,arrows,decorations.markings}
\usepackage{graphicx}

\usepackage{color}
\usepackage{url}
\usepackage{enumerate}
\usepackage[mathscr]{euscript}
  \theoremstyle{plain}

\swapnumbers
    \newtheorem{thm}{Theorem}[section]
    \newtheorem{prop}[thm]{Proposition}
     
   \newtheorem{lemma}[thm]{Lemma}
    \newtheorem{corollary}[thm]{Corollary}
    
    \newtheorem{subsec}[thm]{}
\theoremstyle{definition}
    \newtheorem{defn}[thm]{Definition}
        \newtheorem{remark}[thm]{Remark}
    \newtheorem{exam}[thm]{Example}

\theoremstyle{remark}

\title{}
\author{}
\date{}
\usepackage{amssymb}

\usepackage{hyperref}
\hypersetup{
	colorlinks,
	citecolor=blue,
	filecolor=black,
	linkcolor=blue,
	urlcolor=black
}

\begin{document}
\title{On Equivariant dendriform algebras}
\author{Apurba Das}
\address{Department of Mathematics and Statistics,
Indian Institute of Technology, Kanpur 208016, Uttar Pradesh, India.}
\email{apurbadas348@gmail.com}

\author{Ripan Saha}
\address{Department of Mathematics,
Raiganj University, Raiganj, 733134,
West Bengal, India.}
\email{ripanjumaths@gmail.com}


\begin{abstract}
Dendriform algebras are certain associative algebras whose product splits into two binary operations and the associativity splits into three new identities. In this paper,
we study finite group actions on dendriform algebras. We define equivariant cohomology for dendriform algebras equipped with finite group actions similar to the Bredon cohomology for topological $G$-spaces. We show that equivariant cohomology of such dendriform algebras controls equivariant one-parameter formal deformations.
\end{abstract}
\maketitle
%\tableofcontents
\section{Introduction}
Dendriform algebras were first introduced by Jean-Louis Loday as a Koszul dual algebra of diassociative algebras \cite{loday}. A dendriform algebra is given by a vector space $A$ together with two bilinear maps $\prec, \succ~: A \times A \rightarrow A$ satisfying three new identities (see Definition \ref{defn-dend-alg}). Dendriform algebras are splitting of associative algebras in the sense that the sum operation $\prec + \succ$ in a dendriform algebra turns out to be associative. Dendriform algebras are widely studied in last twenty years as it meets several branches of mathematics, including probability, combinatorics, operads, Lie and Leibniz algebras, arithmetic on planar binary trees and quantum field theory \cite{aguiar , farb-guo, foissy, frab, guo, loday50, ronco}, see also the references therein for more details. In \cite{loday} Loday defines cohomology for a dendriform algebra with trivial coefficients.
The general cohomology theory for algebras over binary quadratic operad was given in \cite{bal}. One may also look at the book by Loday and Vallette \cite{lod-val-book}.
An explicit description for the cohomology of dendriform algebras is given in \cite{Das} using the notion of `multiplicative operad'.

\medskip

Deformations of some algebraic structure were initiated with the seminal work of Gerstenhaber for associative algebras \cite{G2} and subsequently generalized to Lie algebras by Nijenhuis and Richardson \cite{nij-ric}. Deformations of various other algebraic structures were also studied over the years \cite{bal-def , gers-sch1 , gers-sch2, maj-mukh}. In \cite{Das} the author explicitly studied the deformation of dendriform algebras and the governing cohomology.

\medskip

%In this paper, we introduce a notion of a finite group action on dendriform algebra. 
Our aim in this paper is to study dendriform algebras equipped with finite group actions. Such actions are the algebraic version of topological spaces equipped with finite group actions \cite{bredon67}. Let $G$ be a finite group and consider the category $\mathrm{Rep}(G)$ of representations of $G$. A dendriform algebra equipped with a $G$-action is the same as a dendriform algebra in the category $\mathrm{Rep}(G)$. Motivated by the examples of usual dendriform algebras, we provide examples of equivariant dendriform algebras. Namely, we show that Rota-Baxter operators in  $\mathrm{Rep}(G)$ or more generally $\mathcal{O}$-operators in  $\mathrm{Rep}(G)$ induces equivariant dendriform algebras. The tensor module of an object in $\mathrm{Rep}(G)$ also has an equivariant dendriform algebra structure.

\medskip

In the next, we apply the Bredon cohomology approach of topological $G$-spaces to equivariant dendriform algebras. We define cohomology for equivariant dendriform algebras, which can be seen as an equivariant version of the cohomology introduced in \cite{Das}. 
The cochain groups defining the cohomology can be given the structure of an operad with a multiplication. Thus, by a result of Gerstenhaber and Voronov \cite{gers-voro}, the cochain groups inherit a homotopy $G$-algebra structure and the cohomology carries a Gerstenhaber algebra structure.
We show that there is a morphism from the cohomology of equivariant dendriform algebras to the cohomology of corresponding equivariant associative algebras.

\medskip

Finally, we define an equivariant one-parameter formal deformation of equivariant dendriform algebras and show how equivariant cohomology control equivariant deformations of such algebras.
%he vanishing of the equivariant second cohomology implies the rigidity of the structure. The obstruction to extending a finite order deformation turns out to be a $3$-cocycle. As a consequence, the vanishing of the third cohomology ensures that a finite order deformation is always extendable. Thus, in this case, any $2$-cocycle is the linear term of some formal one-parameter deformation of the structure.

\section{Dendriform algebras and its cohomology}\label{sec-2}
In this section, we recall some preliminaries on dendriform algebras, their representations and cohomology. Our main references are \cite{loday, Das}.
\begin{defn} \label{defn-dend-alg}
A dendriform algebra is a $\mathbb{K}$-vector space $A$ together with two $\mathbb{K}$-bilinear maps $\prec, \succ ~: A \times A \rightarrow A$ satisfying
\begin{align}
 (a \prec b) \prec c =~& a \prec (b \prec c + b \succ c), \label{dend-eqn-1}\\
 (a \succ b) \prec c =~&  a \succ (b \prec c), \label{dend-eqn-2}\\
 (a \prec b + a \succ b) \succ c =~& a \succ (b \succ c), \label{dend-eqn-3}
\end{align}
for all $a, b , c \in A.$
\end{defn}

A dendriform algebra as above may be denoted by $(A, \prec, \succ)$ or simply by $A$. Let $A = (A, \prec, \succ)$ and $A' = (A', \prec', \succ')$ be two dendriform algebras.
A morphism between them consists of a linear map $f: A \to A'$ preserving the structure maps, that is,
$f(a \prec b)= f(a) \prec' f(b)$ and
$f(a \succ b)= f(a) \succ' f(b).$

\medskip

Let $C_n$ be the set of first $n$ natural numbers. For convenience, we denote the elements of $C_n$ by $\{ [1], [2], \ldots, [n] \}.$
For any $m , n \geq 1$ and $1 \leq i \leq m,$ we define certain combinatorial maps $R_0 (m; \overbrace{1, \ldots, 1, \underbrace{n}_{i\text{-th place}}, 1, \ldots, 1}^{m}) : C_{m+n-1} \rightarrow C_m$  and  $R_i (m; \overbrace{1, \ldots, 1, \underbrace{n}_{i\text{-th place}}, 1, \ldots, 1}^{m}) : C_{m+n-1} \rightarrow \mathbb{K}[C_n]$   by
\begin{align*} R_0 (m; 1, \ldots, 1, n, 1, \ldots, 1) ([r]) ~=~
\begin{cases} [r] ~~~ &\text{ if } ~~ r \leq i-1 \\ [i] ~~~ &\text{ if } i \leq r \leq i +n -1 \\
[r -n + 1] ~~~ &\text{ if } i +n \leq r \leq m+n -1 \end{cases}
\end{align*}

\begin{align*} R_i (m; 1, \ldots, 1, n, 1, \ldots, 1) ([r]) ~=~
\begin{cases} [1] + [2] + \cdots + [n] ~~~ &\text{ if } ~~ r \leq i-1 \\ [r - (i-1)] ~~~ &\text{ if } i \leq r \leq i +n -1 \\
[1]+ [2] + \cdots + [n] ~~~ &\text{ if } i +n \leq r \leq m+n -1. \end{cases}
\end{align*}

Let $A$ be any vector space. Consider the collection of vector spaces
\begin{align*}
\mathcal{O}(n) := \mathrm{Hom}_{\mathbb{K}} (\mathbb{K}[C_n] \otimes A^{\otimes n}, A), ~~~~ \text{ for } n \geq 1.
\end{align*}
For $f \in \mathcal{O}(m), ~g \in \mathcal{O}(n)$ and $1 \leq i \leq m$, we define
$f \circ_i g \in \mathcal{O}(m+n-1)$ by
\begin{align*}
&(f \circ_i g) ([r]; a_1, \ldots, a_{m+n-1}) \\
&= f \big( R_0 (m; 1, \ldots, n, \ldots, 1)[r]; a_1, \ldots, a_{i-1}, g (R_i (m; 1, \ldots, n, \ldots, 1)[r]; a_i, \ldots, a_{i+n-1}), a_{i+n}, \ldots, a_{m+n-1} \big)
\end{align*}
for $[r] \in C_{m+n-1}$ and $a_1, \ldots, a_{m+n-1} \in A$.

It has been shown in \cite{Das} that the collection of vector spaces $\{ \mathcal{O}(n) \}_{n \geq 1}$ together with the partial compositions $\circ_i$ forms a non-symmetric operad with the identity element $\mathrm{id} \in \mathcal{O}(1)$ is given by $\mathrm{id} ([1]; a) = a,$ for all $a \in A$.

Using partial compositions in an operad, one can define a circle product $\circ : \mathcal{O}(m) \otimes \mathcal{O}(n) \rightarrow \mathcal{O}(m+n-1)$ by
\begin{align}\label{dend-operad-circ}
f \circ g = \sum_{i=1}^m ~(-1)^{(i-1)(n-1)}~ f \circ_i g, ~~~ f \in \mathcal{O}(m),~ g \in \mathcal{O}(n).
\end{align}

%\begin{defn}\label{dend-mul}
Let $(A, \prec, \succ)$ be a dendriform algebra. Define an element $\pi_A \in \mathcal{O}(2) = \mathrm{Hom}_{\mathbb{K}} (\mathbb{K}[C_2] \otimes A^{\otimes 2}, A)$ by
$$ \pi_A ([r]; a, b) = \begin{cases}  a \prec b & \mbox{ if }  [r] = [1]\\
a \succ b & \mbox{ if } [r] = [2].
\end{cases} $$

Then $\pi_A \in \mathcal{O}(2)$ defines a multiplication in the above defined operad, that is, $\pi_A$ satisfies $\pi_A \circ_1 \pi_A = \pi_A \circ_2 \pi_A$ (equivalently $\pi_A \circ \pi_A = 0$).
%\end{defn}
\begin{defn}
A representation of $A$ is given by a vector space $M$ together with two left actions 
\begin{align*}
\prec ~: A \otimes M \rightarrow M \qquad \succ ~: A \otimes M \rightarrow M
\end{align*}
and two right actions
\begin{align*}
\prec ~: M \otimes A \rightarrow M \qquad \succ ~: M \otimes A \rightarrow M
\end{align*}
satisfying the $9$ identities where each pair of $3$ identities correspond to the identities (\ref{dend-eqn-1})-(\ref{dend-eqn-3}) with exactly one of $a, b, c$ is replaced by an element of $M$.
%\begin{align*}
%& (a \prec b) \prec m = a \prec (b \prec m + b \succ m),\\
%& (a \succ b) \prec m =  a \succ (b \prec m),\\
%& (a \prec b + a \succ b) \succ m = a \succ (b \succ m),\\
%& \\
%& (a \prec m) \prec c = a \prec (m \prec c + m \succ c),\\
%& (a \succ m) \prec c =  a \succ (m \prec c),\\
%& (a \prec m + a \succ m) \succ c = a \succ (m \succ c), \\
%& \\
%& (m \prec b) \prec c = m \prec (b \prec c + b \succ c),\\
%& (m \succ b) \prec c =  m \succ (b \prec c),\\
%& (m \prec b + m \succ b) \succ c = m \succ (b \succ c),
%\end{align*}
%for all $a, b, c \in A$ and $m \in M$.
\end{defn}

See \cite{Das} for more details.
%\begin{remark}\label{remark-repn}
To rewrite the $9$ identities of a representation in a more compact form, we define two maps $\theta_1 : \mathbb{K}[C_2] \otimes (A \otimes M) \rightarrow M$ and $\theta_2 : \mathbb{K}[C_2] \otimes (M \otimes A) \rightarrow M$ by

$$ \theta_1 ([r]; a, m) = \begin{cases} a \prec m ~~~ &\mathrm{ if } ~~ [r]=[1] \\
a \succ m ~~~ &\mathrm{ if } ~~ [r] = [2] \end{cases} \qquad  \mathrm{and}  \qquad
 \theta_2 ([r];  m, a) = \begin{cases} m \prec a ~~~ &\mathrm{ if } ~~ [r]=[1] \\
m \succ a ~~~ &\mathrm{ if } ~~ [r] = [2]. \end{cases}$$
Then a representation can be equivalently described as
\begin{align}
\theta_1 \big(  R_0 (2; 1, 2) [s]; ~ a,~ \theta_1 ( R_2 (2; 1, 2)[s]; b, m)  \big) =~& \theta_1 (R_0 (2; 2, 1)[s] ; ~\pi_A (R_1 (2; 2, 1)[s]; a, b),~ m), \label{rep-1}\\
\theta_1 \big(  R_0 (2; 1, 2) [s]; ~ a,~ \theta_2 ( R_2 (2; 1, 2)[s]; m, c)  \big) =~& \theta_2 (R_0 (2; 2, 1)[s] ; ~\theta_1 (R_1 (2; 2, 1)[s]; a, m),~ c), \label{rep-2}\\
\theta_2 \big(  R_0 (2; 1, 2) [s]; ~ m,~ \pi_A ( R_2 (2; 1, 2)[s]; b, c)  \big) =~& \theta_2 (R_0 (2; 2, 1)[s] ; ~\theta_2 (R_1 (2; 2, 1) [s] ; m, b),~ c), \label{rep-3}
\end{align}
for all $[s] \in C_3$ and $a, b, c \in A$, $m \in M$.
%\end{remark}

Any dendriform algebra $A$ is a representation of itself with $\theta_1 = \theta_2 = \pi_A$. In such a case, all the above three identities (\ref{rep-1})-(\ref{rep-3}) are equivalent to $\pi_A \circ_2 \pi_A = \pi_A \circ_1 \pi_A$, which holds automatically as $\pi_A$ defines a multiplication.

Let $(A, \prec, \succ)$ be a dendriform algebra and $M$ be a representation of it. The group $C^n_{\mathrm{dend}} (A, M)$ of $n$-cochains of $A$ with coefficients in $M$  is given by
\begin{align*}
C^n_{\mathrm{dend}} (A, M) := \mathrm{Hom}_\mathbb{K} (\mathbb{K}[C_n] \otimes A^{\otimes n}, M), ~~~ \text{ for } n \geq 1. 
\end{align*}
There are maps $\delta_i : C^n_{\mathrm{dend}} (A, M) \to C^{n+1}_{\mathrm{dend}} (A, M)$, for $0 \leq i \leq n+1$, given as follows
\begin{align*}
&(\delta_i f) ([r]; a_1 , \ldots, a_{n+1}) = \\
&\begin{cases}
                \theta_1 \big( R_0 (2;1,n) [r]; ~a_1, f (R_2 (2;1,n)[r]; a_2, \ldots, a_{n+1})   \big),~~~\text{for }~~ i=0,\\
          f \big(  R_0 (n; 1, \ldots, 2, \ldots, 1)[r]; a_1, \ldots, a_{i-1}, \pi_A (R_i (1, \ldots, 2, \ldots, 1)[r]; a_i, a_{i+1}), a_{i+2}, \ldots, a_{n+1}   \big),~~~\text{for }~~ 1\leq i \leq n,\\
              \theta_2 \big( R_0 (2; n, 1) [r];~ f (R_1 (2;n,1)[r]; a_1, \ldots, a_n), a_{n+1}   \big),~~~\text{for }~~ i=n+1.
               \end{cases}
\end{align*}
Finally, the coboundary map $\delta^n : C^n_{\mathrm{dend}} (A, M) \to C^{n+1}_{\mathrm{dend}} (A, M)$ is defined as 
\begin{align}\label{non-eq-cob}
\delta^n= \sum_{0 \leq i \leq n+1} (-1)^i \delta_i.
\end{align}
The corresponding cohomology groups are denoted by $H^n_{\text{dend}} (A, M)$, for $n \geq 1$. When $M = A$, the cohomology is induced from the multiplication $\pi_A$ on the operad $\{ \mathcal{O}(n) \}_{n \geq 1}$, see \cite{Das} for details.
When $M = \mathbb{K}$ with the trivial representation of $A$ ($\theta_1 = 0$ and $\theta_2 = 0$), the corresponding coboundary operator coincides with the one defined by Loday \cite{loday}.

\section{Equivariant dendriform algebras}
In this section, we study finite group actions on dendriform algebras. Let $G$ be a fixed finite group. We first define the category $\mathrm{Rep}(G)$ of representations of $G$. 
\begin{defn}
A representation of $G$ (over $\mathbb{K}$) consists of a pair $(V, \varrho)$ of a vector space $V$ over $\mathbb{K}$ and a map $\varrho : G \times V \rightarrow V$ satisfying
\begin{itemize}
\item for each $g \in G$, the map $\varrho ( g, ~) : V \rightarrow V$ is linear,
\item $\varrho (e, v) = v$, ~~$\varrho ( g, \varrho (h, v)) = \varrho (gh, v ),$
\end{itemize}
for all $g, h \in G, v \in V$, where $e$ is the neutral element of the group $G$.
\end{defn}

It follows from the above definition that for each $g \in G$, the linear map $\varrho ( g, ~) : V \rightarrow V$ is invertible with inverse $\varrho ( g^{-1}, ~) : V \rightarrow V$.
Note that any group $G$ has a representation on $V = \mathbb{K}[G]$, the vector space generated by the elements of $G$ with $\varrho ( g, \sum_{h \in G} \alpha_h h ) = \sum_{h \in G} \alpha_h (gh).$

\begin{defn}
Let $(V, \varrho_V)$ and $(W, \varrho_W)$ be two representations of $G$. A morphism between them consists of a linear map $\phi : V \rightarrow W$ satisfying
\begin{align*}
\phi \circ \varrho_V ( g, ~) = \varrho_W ( g, ~) \circ \phi, ~~~ \text{ for all } g \in G.
\end{align*}
\end{defn}

For any representation $(V, \varrho)$, the identity map $\text{id}_V : V \rightarrow V$ is a morphism from $(V, \varrho)$ to itself. Let $(V, \varrho_V), (W, \varrho_W)$ and $(U, \varrho_U)$ be three representations of $G$; let $\phi : V \rightarrow W$ and $\psi : W \rightarrow U$ be morphisms between representations. Then the composition $\psi \circ \phi : V \rightarrow U$ is also a morphism between representations.

We define $\text{Rep}(G)$ to be the category whose objects are representations of $G$ and morphisms are given by morphism between representations. The category $\text{Rep}(G)$ is called the category of representations of $G$.

Let $(V, \varrho_V)$ and $(W, \varrho_W)$ be in $\text{Rep}(G)$. Their tensor product is defined by $(V \otimes W, \varrho_{V \otimes W})$ where $\varrho_{V \otimes W}$ is given by
\begin{align*}
\varrho_{V \otimes W} : G \times (V \otimes W) \rightarrow (V \otimes W), ~~~( g, v \otimes w) \mapsto \varrho_V (g, v) \otimes \varrho_W (g, w).
\end{align*}
With respect to the above tensor product, the category $\text{Rep}(G)$ is a monoidal category.

%If $(V' , \varrho_{V'})$ and $(W', \varrho_{W'})$ are also in $\text{Rep}(G)$ and $\phi : V \rightarrow V'$, $\psi : W \rightarrow W'$ are morphism between representations, then $\phi \otimes \psi: V \otimes W \rightarrow V' \otimes W'$ defines a morphism of representations from $(V \otimes W, \varrho_{V \otimes W})$ to $(V' \otimes W', \varrho_{V' \otimes W'})$. Therefore, the tensor product defines a bifunctor $\otimes : \text{Rep}(G) \times \text{Rep}(G) \rightarrow \text{Rep}(G).$ There is also a natural isomorphism $\alpha_{V, W, U} : (V \otimes W) \otimes U \rightarrow V \otimes (W \otimes U)$ between representations given by $\alpha_{V, W, U} ((v \otimes w) \otimes u) \mapsto v \otimes (w \otimes u)$ which satisfies the usual pentagonal diagram. Consider the object $(\mathbb{K}, \varrho_{id})$ in $\text{Rep}(G)$, where $\varrho_{id} : G \times \mathbb{K} \rightarrow \mathbb{K}$ is given by $\varrho_{id} (g, c )= c,$ for all $g \in G, c \in \mathbb{K}.$ Then $\lambda_V : \mathbb{K} \otimes V \rightarrow V, ~ 1 \otimes v \mapsto v$ and $\mu_V : V \otimes \mathbb{K} \rightarrow V, ~ v \otimes 1 \mapsto v$ are natural isomorphisms satisfying the standard triangle diagram. Thus, one obtain the following.

%\begin{thm}
%\end{thm}

%\noindent \underline{ Multilinear maps in the category $\text{Rep}(G)$.}

Let $(V, \varrho_V)$ and $(W, \varrho_W)$ be in $\text{Rep}(G)$. Consider the tensor product representation $(V^{\otimes n}, \varrho_{V^{\otimes n}})$ given by $ \varrho_{V^{\otimes n}} : G \times  V^{\otimes n} \rightarrow  V^{\otimes n}$,
\begin{align*}
\varrho_{V^{\otimes n}} ( g, v_1 \otimes \cdots \otimes v_n) = \varrho_V (g, v_1) \otimes \cdots \otimes \varrho_V (g, v_n).
\end{align*}
A morphism from the representation $(V^{\otimes n}, \varrho_{V^{\otimes n}})$ to $(W, \varrho_W)$ is called a multilinear map in $\text{Rep}(G)$. Thus, it is given by a linear map $\phi : V^{\otimes n} \rightarrow W$ satisfying 
\begin{align*}
\phi ( \varrho_V (g, v_1) \otimes \cdots \otimes \varrho_V (g, v_n)) = \varrho_W ( g, \phi (v_1 \otimes \cdots \otimes v_n)).
\end{align*}
In this case, we write $\phi \in \text{Hom}_{\text{Rep}(G)} (V^{\otimes n}, W).$\\

%\noindent \underline{ Gerstenhaber's $\circ_i$-operations in $\text{Rep}(G)$.}

\subsection{Dendriform algebras and Rota-Baxter operators in Rep(G)}

\begin{defn}\label{equivariant dend}
A dendriform algebra in $\text{Rep}(G)$ consists of an object $V = (V, \varrho_V)$ in $\text{Rep}(G)$ together with morphisms $\prec, \succ \in \text{Hom}_{\text{Rep}(G)} (V^{\otimes 2}, V)$ satisfying the usual dendriform identities
\begin{align}
 (a \prec b) \prec c =~& a \prec (b \prec c + b \succ c), \label{dend-eqne-1}\\
 (a \succ b) \prec c =~&  a \succ (b \prec c), \label{dend-eqne-2}\\
 (a \prec b + a \succ b) \succ c =~& a \succ (b \succ c), ~~~ \text{ for all } a, b , c \in V. \label{dend-eqne-3}
\end{align}
\end{defn}

A dendriform algebra in $\text{Rep}(\{e\})$ is the usual dendriform algebra. It follows from (\ref{dend-eqne-1})-(\ref{dend-eqne-3}) that the sum $\star = \prec+ \succ \in \text{Hom}_{\text{Rep}(G)} (V^{\otimes 2}, V)$ defines an associative algebra structure on $V = (V, \varrho_V)$ in $\text{Rep}(G)$. In \cite{MY19} the authors call such an object an equivariant associative algebra.

\begin{defn}
Let $V = (V, \varrho_V)$ be an associative algebra in $\text{Rep}(G)$ with multiplication $\mu \in \text{Hom}_{\text{Rep}(G)} (V^{\otimes 2}, V)$. A morphism $R \in \text{Hom}_{\text{Rep}(G)} (V, V)$ is said to be a Rota-Baxter operator (of weight $0$) on $V$ if $R$ satisfies
\begin{align*}
\mu (R(a), R(b) ) = R ( \mu (a, Rb) + \mu (Ra, b) ), ~~~ \text{ for all } a, b \in V.
\end{align*}
\end{defn}

In \cite{aguiar} Aguiar showed that a Rota-Baxter operator on an associative algebra induces a dendriform structure. Here we extend this result in the category $\text{Rep}(G)$.

\begin{exam} Let $R \in \text{Hom}_{\text{Rep}(G)} (V, V)$ be a Rota-Baxter operator on the associative algebra $(V, \mu)$ in $\mathrm{Rep}(G)$.
Then $R$ induces a dendriform structure on $V = (V, \varrho_V)$ with
\begin{align*}
a \prec b = \mu ( a, Rb)    \quad \text{and } \quad a \succ b = \mu ( Ra, b).
\end{align*}
\end{exam}

Thus it follows that a Rota-Baxter operator $R$ induces a new associative structure on $(V, \varrho_V)$ given by
\begin{align}\label{rota-ass}
\overline{\mu} (a, b ) =  \mu ( a, Rb) + \mu (Ra, b).
\end{align}
This associative structure on $(V, \varrho_V)$ can also be described using weak pseudotwistor in $\mathrm{Rep} (G).$
Weak pseudotwistors were studied in the context of usual dendriform algebras by Brzezi\'{n}ski \cite{bere}.

\begin{defn}
Let $(V, \mu)$ be an associative algebra in $\text{Rep}(G)$. A morphism $T \in \text{Hom}_{\mathrm{Rep}(G)} (V^{\otimes 2}, V^{\otimes 2} )$ is called a weak pseudotwistor if there exists a morphism $\tau \in \text{Hom}_{\mathrm{Rep}(G)} (V^{\otimes 3}, V^{\otimes 3})$ such that the following diagram commute
\[
\xymatrixrowsep{0.5in}
\xymatrixcolsep{0.7in}
\xymatrix{
V^{\otimes 3} \ar[r]^{\mathrm{id}_V \otimes (\mu \circ T)} \ar[d]_\tau & V^{\otimes 2} \ar[d]_{T} & V^{\otimes 3} \ar[d]^{\tau} \ar[l]_{(\mu \circ T) \otimes \mathrm{id}_V} \\
V^{\otimes 3} \ar[r]_{\mathrm{id}_V \otimes \mu} &  V^{\otimes 2}  &  V^{\otimes 3} \ar[l]^{\mu \otimes \mathrm{id}_V}
}
\]
The morphism $\tau$ is called a weak companion of $T$.
\end{defn}
 Note that the identity map $T = \mathrm{id}_{V^{\otimes 2}}$ is a weak pseudotwistor with weak companion $\tau = \mathrm{id}_{V^{\otimes 3}}.$

\begin{prop}\label{twi}
Let $T \in \mathrm{Hom}_{\mathrm{Rep}(G)} (V^{\otimes 2}, V^{\otimes 2})$ be a weak pseudotwistor with weak companion $\tau$. Then $\mu \circ T \in  \mathrm{Hom}_{\mathrm{Rep}(G)} (V^{\otimes 2}, V)$ defines a new associative structure on $V$ in $\mathrm{Rep} (G).$
\end{prop}

In the following, we shall construct a weak pseudotwistor out of a Rota-Baxter operator.

\begin{prop}
Let $R \in \mathrm{Hom}_{\mathrm{Rep}(G)} (V, V)$ be a Rota-Baxter operator on an associative algebra $(V, \mu)$. Then the map $T \in \mathrm{Hom}_{\mathrm{Rep}(G)} (V^{\otimes 2}, V^{\otimes 2})$ given by
\begin{align*}
T(a, b) = a \otimes R(b) + R(a) \otimes b
\end{align*}
is a weak pseudotwistor with weak companion $\tau ( a \otimes b \otimes c) = R(a) \otimes R(b) \otimes c + R(a) \otimes b \otimes R(c) + a \otimes R(b) \otimes R(c).$
\end{prop}

 \begin{proof}
Note that \begin{align*}
 &\big( T \circ (\mathrm{id}_V \otimes ( \mu \circ T)) \big ) (a, b, c) \\
 &= T ( a, b R(c) + R(b) c ) \\
 &= a \otimes R (b R(c)) + a \otimes R ( R(b) c ) + R(a) \otimes b R(c) + R(a) \otimes R(b) c
 \end{align*}
 and
 \begin{align*}
 &\big( (\mathrm{id}_V \otimes \mu ) \circ \tau \big)(a, b, c) \\
 &= (\mathrm{id}_V  \otimes \mu ) \big( R(a) \otimes R(b) \otimes c + R(a) \otimes b \otimes R(c) + a \otimes R(b) \otimes R(c) \big) \\
 &= R(a) \otimes R(b) c + R(a) \otimes b R(c) + a \otimes R(b) R(c).
 \end{align*}
 Since $R$ is a Rota-Baxter operator, we have $T \circ (\mathrm{id}_V \otimes ( \mu \circ T)) = (\mathrm{id}_V  \otimes \mu ) \circ \tau $. Similarly, one can show that $T \circ ((\mu \circ T) \otimes \mathrm{id}_V) = (\mu \otimes \mathrm{id}_V) \circ \tau$. Hence the proof.
\end{proof}

Thus, it follows from Proposition \ref{twi} that $V$ inherits a new associative algebra with $\overline{\mu} (a \otimes b ) = \mu ( a, R(b)) + \mu (R(a) \otimes b).$ This gives an alternative view of the associative product (\ref{rota-ass}) in terms of weak pseudotwistor.

\medskip

Next we introduce a more general operator and construct dendriform algebra from that.  Let $(V, \mu)$ be an associative algebra in $\text{Rep}(G)$. A bimodule over it consists of an object $(M, \varrho_M)$ in $\text{Rep}(G)$ together with morphisms $l \in \text{Hom}_{\mathrm{Rep}(G)} ( V \otimes M, M)$ and $r \in \text{Hom}_{\mathrm{Rep}(G)} (M \otimes V, M)$ satisfying the usual bimodule conditions
\begin{align*}
l (\mu (a, b), m) = l (a, l(b, m)),~~~~~ r (l(a,m), b) = l (a, r(m, b)) ~~~~ \text{ and } ~~~ r(r(m, a), b ) = r (m, \mu (a, b )),~~~
\end{align*}
for all $a, b \in V$ and  $m \in M$. 

With respect to the above notations, we have the following. The proof is similar to the classical case.

\begin{prop} (Semi-direct product)
Let $(V, \mu)$ be an associative algebra in ${\mathrm{Rep}(G)}$ and $(M, l, r)$ a bimodule over it. Then the direct sum $(V \oplus M, \varrho_V \oplus \varrho_M)$ inherits an associative structure in ${\mathrm{Rep}(G)}$ with product
\begin{align*}
\mu' ((a,m)\otimes (b,n)) = (\mu (a, b), l (a,n) + r (m, b)), ~~~ \text{ for } (a,m), (b, n) \in V \oplus M. 
\end{align*}
\end{prop}

\begin{defn}
An $\mathcal{O}$-operator on $(V, \mu)$ with respect to a bimodule $(M, l, r)$ is a 
morphism $T \in \mathrm{Hom}_{\mathrm{Rep}(G)} (M, V)$ satisfying
\begin{align*}
\mu (T(m), T(n) ) = T ( r (m, Tn) + l (Tm, n)), ~~~ \text{ for all } m, n \in M.
\end{align*}
\end{defn}

The notion of $\mathcal{O}$-operators (also called generalized Rota-Baxter operators) on usual associative algebras was first introduced by Uchino as an associative analogue of Poisson structures \cite{uchino}. Therefore, $\mathcal{O}$-operators in $\mathrm{Rep}(G)$ can be thought of as an associative analogue of equivariant Poisson structures. Note that a Rota-Baxter operator in $\mathrm{Rep}(G)$ is a $\mathcal{O}$-operator with respect to the adjoint bimodule.

$\mathcal{O}$-operators can be characterized in terms of its graph. The following is a generalization of \cite[Lemma 2.7]{uchino}.

\begin{prop}
Let $(V, \mu)$ be an associative algebra in ${\mathrm{Rep}(G)}$ and $(M, l, r)$ a bimodule over it. A morphism $T \in \mathrm{Hom}_{\mathrm{Rep}(G)} (M, V)$ is an $\mathcal{O}$-operator if and only if 
\begin{align*}
\mathrm{Gr} (T) = \{ (Tm, m) |~ m \in M \} \subset V \oplus M
\end{align*} 
is a subalgebra, where $V \oplus M$ is equipped with the semi-direct product algebra structure.
\end{prop}

\begin{exam}
Let $T$ be an $\mathcal{O}$-operator on the associative algebra $(V, \mu)$ with respect to the bimodule $(M, l, r)$, then $M = (M, \varrho_M)$ inherits a dendriform algebra structure in $\text{Rep}(G)$ with
\begin{align*}
m \prec n = r (m, Tn)  \quad \text{ and } \quad  m \succ n = l (Tm, n), ~~~ \mathrm{ for }~~ m, n \in M.
\end{align*}
\end{exam}

Here one may ask the following question: whether every dendriform algebra in $\mathrm{Rep}(G)$ is induced from an $\mathcal{O}$-operator? Let $(A, \prec, \succ)$ be a dendriform algebra in $\mathrm{Rep}(G)$ with the corresponding associative algebra $(A, \prec + \succ )$ in $\mathrm{Rep}(G)$. We denote this associative algebra by $A_{\mathrm{ass}}.$ Then $A$ is an $A_{\mathrm{ass}}$-bimodule with $l (e, a) = e \succ a$ and $r (a, e) = a \prec e$, for $e \in A_{\mathrm{ass}}$, $a \in A$. With respect to this $A_{\mathrm{ass}}$-bimodule, it is easy to see that the identity map $\mathrm{id} : A \rightarrow A_{\mathrm{ass}}$ is an $\mathcal{O}$-operator. Moreover, the induced dendriform structure on $A$ is the given one. This shows that the answer is affirmative that asked at the beginning of this paragraph.
 
Another example of a dendriform algebra in $\mathrm{Rep}(G)$ arises from the tensor module of a group representation.

\begin{exam} (Tensor module in $\text{Rep}(G)$) Let $V = (V, \varrho_V) \in \mathrm{Rep}(G)$. Consider $(\overline{T}(V), \varrho_{\overline{T}(V)}) \in \text{Rep}(G)$ where $\overline{T}(V) = \oplus_{n \geq 1} V^{\otimes n}$ and 
\begin{align*}
\varrho_{\overline{T}(V)}|_{G \times V^{\otimes n}} = \varrho_{V^{\otimes n}} : G \times V^{\otimes n} \rightarrow V^{\otimes n}.
\end{align*}
Define 
\begin{align*}
(v_1 \otimes \cdots \otimes v_n) \prec (v_{n+1} \otimes \cdots \otimes v_{n+m} ) := \sum_{\sigma \in \mathrm{Sh}^1_{n,m}} v_{\sigma^{-1}(1)} \otimes \cdots \otimes v_{\sigma^{-1}(n+m)}, \\
(v_1 \otimes \cdots \otimes v_n) \succ (v_{n+1} \otimes \cdots \otimes v_{n+m} ) := \sum_{\sigma \in \mathrm{Sh}^2_{n,m}} v_{\sigma^{-1}(1)} \otimes \cdots \otimes v_{\sigma^{-1}(n+m)},
\end{align*}
where $\mathrm{Sh}^1_{n,m}$ and $\mathrm{Sh}^2_{n,m}$ are the disjoint exhaustive subsets of $\mathrm{Sh}_{n,m}$ given by
\begin{align*}
\mathrm{Sh}^1_{n,m} = \{ \sigma \in \mathrm{Sh}_{n, m} |~ \sigma (n) = n+m \} ~~\text{ and }~~ \mathrm{Sh}^2_{n,m} = \{ \sigma \in \mathrm{Sh}_{n, m} |~ \sigma (n+m) = n+m \}.
\end{align*}
Then $(\prec, \succ)$ defines a dendriform structure on $(\overline{T}(V), \varrho_{\overline{T}(V)})$.
\end{exam}

\subsection{Equivariant dendriform algebra in Bredon's sense}
Dendriform algebras in $\mathrm{Rep}(G)$ can also be seen as dendriform algebras equipped with an action of $G$. Therefore, they are algebraic version of topological $G$-spaces in the sense of Bredon \cite{bredon67}.
\begin{defn}\label{definition-group-action}
Let $A = (A, \prec, \succ)$ be a dendriform algebra and $G$ be a finite group. The group $G$ is said to act from the left on $A$ if there exists a map
$$\psi : G\times A \rightarrow A,~~ (g, a) \mapsto \psi (g, a) = ga$$ satisfying the following conditions
\begin{enumerate}
\item for each $g\in G$, the map $\psi_g : A \rightarrow A,~ a \mapsto ga$ is linear,
\item $ea= a$ for all $a \in A$, where $e \in G$ is the neutral element,
\item $g(ha) = (gh)a$ for all $g, h \in G$ and $a \in A$,
\item $g(a \prec b)= ga \prec gb$ and  $g(a \succ b)= ga \succ gb$, for all $g\in G$ and $a, b \in A.$
\end{enumerate}
We use the notation $(G, A)$ to denote a dendriform algebra $A$ equipped with an action of $G$.
\end{defn}
Let $(G, A)$ and $(G, A')$ be two dendriform algebras equipped with actions of $G$. A morphism between  $(G, A)$ and $(G, A')$ is a morphism of dendriform algebras
$f: A\to A'$
that satisfies $f(ga)=g f(a)$, for all $g \in G, ~a\in A$.
%\begin{example}
%Let $(A,\mu)$ be an associative algebra equipped with an action of finite group $G$ \cite{MY19} and $R: A \to A$ is a linear map satisfying
%\begin{enumerate}
%\item $\mu(R(a) , R(b)) = R \big( \mu(a , R(b)) + \mu(R(a) , b) \big), ~~~ \text{ for all } a, b \in A.$
%\item $R(ga)=g R(a),~~~ \text{ for all } a \in A,~~ g\in G.$
%\end{enumerate}
%Then $(A, \mu, \prec, \succ)$ is a dendriform algebra with a natural action of $G$ on $A$, where $\prec, \succ$ is defined as
%\begin{align*}
%a \prec b = \mu(a,  R (b)) ~~~~~~~~ \mathrm{ and } ~~~~~~~~ a \succ b = \mu(R(a), b).
%\end{align*}
%\end{example}
The following is an equivalent formulation of the above definition.
\begin{prop}
Let $G$ be a finite group and $A$ be a dendrifrom algebra. Then $G$ acts on $A$ if and only if there exists a group homomorphism 
$G \rightarrow \mathrm{Iso}_{\mathrm{Dend}} (A, A)$ from the group $G$ to the group $ \mathrm{Iso}_{\mathrm{Dend}} (A, A)$ of dendriform algebra isomorphisms on $A$.
\end{prop}
%\begin{remark}
%Let ${\mathbb K}[G]$ be the group ring. If a dendriform algebra $A$ is equipped with an action of $G$ then $A$ may be viewed as a ${\mathbb K}[G]$-module.
%\end{remark}
Let $A$ be a dendriform algebra equipped with an action of  $G$. Given a subgroup $H \subset G,$ the $H$-fixed point set $A^H$ is defined by
$$A^H = \{a \in A |~ ha = a, ~\forall h \in H\}.$$ 
\begin{lemma}
Let $(A,\prec, \succ)$ be a dendriform algebra equipped with an action of $G$. For every subgroup $H \subset G,$ $A^H$ is a dendriform subalgebra of $A$.
\end{lemma}
\begin{proof}
To prove $A^H$ is a dendriform subalgebra, it is enough to show $A^H$ is closed under the bilinear maps $\prec$ and $\succ$. Let $a, b \in A^H$ and $h\in H$. Then we have
\begin{align*}
 h(a \prec b)= ha \prec hb=a \prec b ~~~ \text{ and } ~~~ h(a \succ b)= ha \succ hb=a \succ b.
\end{align*}
This shows that $a \prec b$ and $a \succ b\in A^H$. Therefore, $A^H$ is a dendriform subalgebra.
\end{proof}

%\begin{prop}
%Considering dendriform algebra $A$ as a representation over itself, maps $\pi_A, \theta_1, \theta_2$ respects the group action, in the sense that,
%$$\pi_A ([r]; ga, gb)=g \pi_A ([r]; a, b),~~~ \theta_i ([r]; ga, gm)=g \theta_i ([r]; a, m)~~~\text{for}~i=1, 2.$$
%\end{prop}
%\begin{proof}
%Follows from the definitions of $\pi_A, \theta_1, \theta_2.$
%\end{proof}

\section{Equivariant cohomology of dendriform algebras}
In this section, we shall consider appropriate cohomology for equivariant dendriform algebras. We give two equivalent formulations of the cohomology, one in $\mathrm{Rep}(G)$ following \cite{Das} and another following Bredon's cohomology for topological $G$-spaces \cite{bredon67}. 

%Let $C_n$ be the set of first $n$ natural numbers. Since we will treat the elements of $C_n$ as certain symbols, we denote the elements of $C_n$ by $C_n = \{ [1], [2], \ldots, [n] \}.$ There are maps $R_0 (m; \overbrace{1, \ldots, 1, \underbrace{n}_{i\text{-th place}}, 1, \ldots, 1}^{m}) : C_{m+n-1} \rightarrow C_m$ and $R_i (m; \overbrace{1, \ldots, 1, \underbrace{n}_{i\text{-th place}}, 1, \ldots, 1}^{m}) : C_{m+n-1} \rightarrow \mathbb{K}[C_n]$ defined by

%\begin{align*} R_0 (m; 1, \ldots, 1, n, 1, \ldots, 1) ([r]) ~=~
%\begin{cases} [r] ~~~ &\text{ if } ~~ r \leq i-1 \\ [i] ~~~ &\text{ if } i \leq r \leq i +n -1 \\
%[r -n + 1] ~~~ &\text{ if } i +n \leq r \leq m+n -1 \end{cases}
%\end{align*}
%We also define maps $R_i (m; \overbrace{1, \ldots, 1, \underbrace{n}_{i\text{-th place}}, 1, \ldots, 1}^{m}) : C_{m+n-1} \rightarrow \mathbb{K}[C_n]$ by
%\begin{align*} R_i (m; 1, \ldots, 1, n, 1, \ldots, 1) ([r]) ~=~
%\begin{cases} [1] + [2] + \cdots + [n] ~~~ &\text{ if } ~~ r \leq i-1 \\ [r - (i-1)] ~~~ &\text{ if } i \leq r \leq i +n -1 \\
%[1]+ [2] + \cdots + [n] ~~~ &\text{ if } i +n \leq r \leq m+n -1. \end{cases}
%\end{align*}

Let $(V, \varrho_V)$ be an element in $\text{Rep}(G)$. For each $n \geq 1$, we define $C^n_{G} (V, V) = \text{Hom}_{\mathrm{Rep}(G)} (\mathbb{K}[C_n] \otimes V^{\otimes n}, V )$, where the representation of $G$ on
$\mathbb{K}[C_n] \otimes V^{\otimes n}$ is given by $\mathrm{id}_{\mathbb{K}[C_n]} \otimes \varrho_{V^{\otimes n}}$. 

The following is a generalization of \cite[Proposition 2.2]{Das}.
\begin{prop}
The collection of vector spaces $\{ C^n_{G} (V, V) \}_{n \geq 1}$ together with partial compositions $\circ_i : C^m_{G} (V, V) \otimes C^n_{G} (V, V) \rightarrow C^{m+n-1}_{G} (V, V)$ defined by
\begin{align*}
&(f \circ_i g) ([r]; a_1, \ldots, a_{m+n-1}) \\
&= f \big(R_0 (m; 1, \ldots, n, \ldots, 1)[r];~ a_1, \ldots, a_{i-1}, \\& \hspace*{2cm} g (R_i (m; 1, \ldots, n, \ldots, 1)[r]; a_i, \ldots, a_{i+n-1}), a_{i+n}, \ldots, a_{m+n-1} \big)
\end{align*}
and the identity map $\mathrm{id}_V \in  C^1_{G} (V, V)$ as identity element forms an operad.
\end{prop}

Note that a multiplication $\pi \in C^2_{G} (V,V) = \text{Hom}_{\mathrm{Rep}(G)} (\mathbb{K}[C_2] \otimes V^{\otimes 2}, V ) $ on this operad is same as an equivariant dendriform algebra structure on $(V, \varrho_V)$. We define the cohomology of an equivariant dendriform algebra as the cohomology induced from the corresponding multiplication in the above operad.

More precisely, the coboundary map $\delta :  C^n_{G} (V, V) \rightarrow  C^{n+1}_{G} (V, V)$ is given by
\begin{align*}
(\delta f)_{[r]} =~& \pi_{R_0 (2;1,n)[r]} \circ (\text{id}_V \otimes f_{R_2 (2;1,n)[r]}) \\
~&+ \sum_{i=1}^n (-1)^i~ f_{R_0 (n;1, \ldots, 2, \ldots,1)[r]} \circ (\text{id}_{V^{\otimes i-1}} \otimes \pi_{R_i (n;1, \ldots, 2, \ldots,1)[r]} \otimes \text{id}_{V^{\otimes n-i}} ) \\
~&+ (-1)^{n+1}~ \pi_{R_0 (2, n, 1)[r]} \circ ( f_{R_1 (2;n, 1)[r]} \otimes \text{id}_V ),
\end{align*}
for $f \in C^n_G (V, V)$ and $[r] \in C_{n+1}.$
Then $\delta^2 = 0$. The cohomology of the cochain complex $\{ C^*_{G} (V, V), \delta \}$ is called the cohomology of the equivariant dendriform algebra. We denote the $n$-th  equivariant cohomology of dendriform algebra $V$ in $\mathrm{Rep}(G)$ as $H^n_G(V, V)$. 
It follows that the graded space of cohomology $ \bigoplus_n H^n_G (V, V)$ carries a Gerstenhaber structure.
One may also define representations and cohomology with coefficients in representation for dendriform algebras in $\mathrm{Rep}(G)$. They are generalizations of the results of Section \ref{sec-2}.

\medskip

In the following, we relate the above cohomology with the Hochschild cohomology of equivariant associative algebras \cite{MY19}.

Let $(V, \varrho_V) \in \text{Rep}(G)$. For any $\phi \in  \text{Hom}_{\text{Rep}(G)} (V^{\otimes m }, V)$ and  $ \psi \in \text{Hom}_{\text{Rep}(G)} (V^{\otimes n}, V)$, we define $\phi \circ_i \phi \in  \text{Hom}_{\text{Rep}(G)} (V^{\otimes m+n-1}, V)$ by
\begin{align*}
(\phi \circ_i \psi) ( v_1, \ldots, v_{m+n-1}) = \phi ( v_1, \ldots, v_{i-1}, \psi (v_i, \ldots, v_{i+n-1}), v_{i+n}, \ldots, v_{m+n-1}).
\end{align*}
There is a degree $-1$ graded Lie bracket on $\bigoplus_n  \text{Hom}_{\text{Rep}(G)} (V^{\otimes n}, V)$  given by
\begin{align*}
[\phi, \psi] = \sum_{i=1}^m (-1)^{(i-1)(n-1)} \phi \circ_i \psi - (-1)^{(m-1)(n-1)} \sum_{i=1}^n (-1)^{(i-1)(m-1)} \psi \circ_i \phi.
\end{align*}
This is the Gerstenhaber bracket on the space of multilinear maps on $(V, \varrho_V)$.

\begin{remark}
The collection of vector spaces $\{ \text{Hom}_{\text{Rep}(G)} (V^{\otimes n}, V)\}_{n \geq 1}$ together with the $\circ_i$-products forms an operad. When $G = \{ e \}$, the group consisting of only one element, this operad reduces to the endomorphism operad on $V$ \cite{lod-val-book}.
\end{remark}

Let $(V, \mu)$ be an associative algebra in $\text{Rep}(G)$. For any $n \geq 0$, the $n$-th Hochschild  cochain group $C^n_{\mathrm{Hoch}} (V, V)$ is given by $C^0_{\mathrm{Hoch}}(V, V) = V^\varrho = \{ v \in V |~ \varrho (g, v ) = v, \forall g \in G \}$ and $C^n_{\mathrm{Hoch}} (V, V) = \mathrm{Hom}_{\mathrm{Rep}(G)} (V^{\otimes n}, V)$, for $n \geq 1$. The coboundary map $\delta : C^n_{\mathrm{Hoch}} (V, V) \rightarrow C^{n+1}_{\mathrm{Hoch}} (V, V)$ is given by
\begin{align}
&(\delta v) (w) = \mu (w, v) - \mu (v, w),~~~~ \text{ for } v \in V^\varrho,\\
&\delta f = \mu \circ (\mathrm{id}_V \otimes f) + \sum_{i=1}^n (-1)^i f \circ (\mathrm{id}_{V^{\otimes i-1}} \otimes f \otimes \mathrm{id}_{V^{\otimes n-i}}) + (-1)^{n+1} \mu \circ (f \otimes \mathrm{id}_V). \label{hoch-diff-e}
\end{align}
Then $\delta^2 = 0$. The cohomology of the complex $\{ C^*_{\mathrm{Hoch}}(V,V), \delta \}$ is called the Hochschild cohomology of the associative algebra $(V, \mu)$ in $\text{Rep}(G).$ We denote the cohomology by $H^n_{\mathrm{Hoch}}(V,V).$

\begin{remark}
Let $(V, \varrho_V)$ be an object in $\text{Rep}(G)$. A morphism $\mu \in \mathrm{Hom}_{\mathrm{Rep}(G)} (V^{\otimes 2}, V)$ defines an associative structure on $V$ if and only if $\mu $ satisfies $\mu \circ_1 \mu = \mu \circ_2 \mu$. In other words, $\mu$ is a multiplication on the operad $\{ \mathrm{Hom}_{\mathrm{Rep}(G)} (V^{\otimes n}, V ), \circ_i \}_{n \geq 1}$ in the sense of \cite{gers-voro}. Therefore, $\mu$ induces a differential 
\begin{align*}
\delta_\mu f = [ \mu , f ]
\end{align*}
on the graded space $\bigoplus_n \mathrm{Hom}_{\mathrm{Rep}(G)} (V^n, V)$. This differential is same as the Hochschild differential (\ref{hoch-diff-e}) up to a sign. Therefore, it follows from Gerstenhaber and Voronov \cite{gers-voro} that the Hochschild cohomology of an associative algebra in $\text{Rep}(G)$ carries a Gerstenhaber structure. See also \cite{G1}.
\end{remark}

\begin{thm}
The collection of maps $$S_n : \mathrm{Hom}_{\mathrm{Rep}(G)} (\mathbb{K}[C_n] \otimes V^{\otimes n} , V ) \rightarrow \mathrm{Hom}_{\mathrm{Rep}(G)} ( V^{\otimes n} , V ), ~f \mapsto f_{[1]} + \cdots + f_{[n]}, \text{ for } n \geq 1$$ is a morphism between operads.

Let $(V, \prec, \succ)$ be a dendriform algebra in $\mathrm{Rep}(G)$ with the corresponding associative algebra $(V, \prec + \succ)$. Then $\{S_n\}$ preserves corresponding multiplications. Hence, they induces a morphism $S_* : H^*_{G} (V, V) \rightarrow H^*_{\mathrm{Hoch}} (V, V)$ between cohomologies. In fact, $S_* $ is a morphism between Gerstenhaber algebras.
\end{thm}

The following is an equivalent formulation of equivariant cohomology of dendriform algebras along the line of Bredon cohomology for a $G$-space \cite{bredon67, elm}.

Let $G$ be a finite group. Recall that the category of canonical orbits of $G$, denoted by $O_G,$ is a category whose objects are left cosets $G/H$, as $H$ runs over all subgroups of $G$. Note that the group $G$ acts on the set $G/H$ by left translation. A morphism from $G/H$ to $G/ K$ is a $G$-map. Recall that such a morphism determines and is determined by a subconjugacy relation $g^{-1}Hg\subseteq K$ and is given by $\hat{g}(eH)=gK$. We denote this morphism by $\hat{g}$ \cite{bredon67}.

We denote the category of dendriform algebras as \textbf{Dend}. 
\begin{defn}
An $O_G$-dendriform algebra is a contravariant functor
$$\mathcal{A} : O_G \to \textbf{Dend}.$$
\end{defn}
\begin{exam}
Let $A$ be a dendriform algebra equipped with an action of $G$.  Then there is a contravariant functor
\begin{align*}
\Phi_A : O_G \to \textbf{Dend}
\end{align*}
defined by $\Phi_A (G/H)=A^H$ and for a morphism $\hat{g}:G/H\to G/K$ corresponding to the subconjugacy relation $g^{-1}Hg\subseteq K$,
$$\Phi_A(\hat{g})=\psi_g: A^K \to A^H.$$
Thus, $\Phi_A$ is an $O_G$-dendriform algebra, which will be referred to as the $O_G$-dendriform algebra associated to $A$.
\end{exam}
\begin{defn}
An $O_G$-module is a contravariant functor
$M: O_G \to \textbf{Mod}$,
where \textbf{Mod} denotes the category of $\mathbb{K}$-modules.
\end{defn}

A representation of an $O_G$-dendriform algebra $\mathcal{A}$ is an $O_G$-module $M$ together with the following natural transformations
\begin{align*}
\alpha_l, \alpha_r : \mathcal{A}\otimes M \to M,\\
\beta_l, \beta_r : M\otimes \mathcal{A} \to M,
\end{align*}
such that for all $G/H \in \text{Ob}(O_G)$, $M(G/H)$ is a non-equivariant representation of $\mathcal{A}(G/H)$ with associated natural transformations $\alpha_l, \alpha_r, \beta_l, \beta_r.$
\begin{remark}
From the above definition of the representation in an equivariant setting we have the following relations coming from the naturality of $\alpha_l,\alpha_r,\beta_l$ and $\beta_r$
\begin{align}
\begin{cases}
\alpha_i(G/H)\circ (\mathcal{A}(\hat{g})\otimes M(\hat{g}))=M(\hat{g})\circ \alpha_i(G/K),\\
\beta_i(G/H)\circ (M(\hat{g})\otimes \mathcal{A}(\hat{g}))=M(\hat{g})\circ\beta_i(G/K)),
\end{cases}
\end{align}
for $i=l,r$.
\end{remark}

We set $$S^n(\mathcal{A};M):=\bigoplus_{H\leq G}C^n_{\text{dend}}(\mathcal{A}(G/H),M(G/H)), ~ \text{ for } n \geq 1$$
and define $\delta^n: S^n(\mathcal{A};M)\to S^{n+1}(\mathcal{A};M)$ by $\delta^n:=\bigoplus_{H\leq G}\delta^n_H$,
where $\delta^H_n: C^n_{\text{dend}}(\mathcal{A}(G/H),M(G/H))\to C^{n+1}_{\text{dend}}(\mathcal{A}(G/H), M(G/H))$ is the non-equivariant coboundary map given by (\ref{non-eq-cob}).

Clearly, $\lbrace S^\sharp(\mathcal{A};M),\delta\rbrace$ is a cochain complex. Throughout this paper, we take the $O_G$-dendriform algebra $\mathcal{A}$ as $\Phi_A$ and consider the cohomology of $\Phi_A$ with coefficients in $\Phi_A$. We define a subcomplex of this cochain complex as follows:
\begin{defn}
A cochain $c=\lbrace c_H \mid H\leq G\rbrace$ is said to be invariant if for every morphism $\hat{g}: G/H\to G/K$ corresponding to a subconjugacy relation $g^{-1}Hg\subseteq K$, following relation holds
\begin{align}\label{inv-defn}
c_H\circ(\mathrm{id} \otimes\mathcal{A}(\hat{g})^{\otimes n})=\mathcal{A}(\hat{g})\circ c_K.
\end{align}
As $\Phi_A$ is an $O_G$-dendrifrom algebra corresponding to the dendriform algebra $A$, the identity (\ref{inv-defn}) is same as
\begin{align*}
c_H\circ(\mathrm{id} \otimes\psi_g^{\otimes n})=\psi_g\circ c_K.
\end{align*}
\end{defn}
\begin{lemma}
The set of all invariant $n$-cochains $S_G^n(\Phi_A;\Phi_A)$ is a subgroup of $S^n(\Phi_A;\Phi_A)$. In other words, if $c=\lbrace c_H\rbrace\in S^n_G(\Phi_A;\Phi_A)$ is an invariant cochain then $\delta^n_H(c)=\lbrace \delta^n_H(c_H)\rbrace\in S^{n+1}_G (\Phi_A;\Phi_A)$ is an invariant $(n+1)$-cochain.
\end{lemma}
\begin{proof}
Suppose $c=\lbrace c_H\rbrace\in S^n_G(\Phi_A;\Phi_A)$ is an invariant cochain and $\hat{g}: G/H\to G/K$ is a morphism in $O_G$, which corresponds to the subconjugacy relation $g^{-1}Hg\subseteq K$. Thus, from the naturality of functors, we have
$$c_H\circ(\mathrm{id} \otimes\psi_g^{\otimes n})=\psi_g\circ c_K.$$
Now,
\begin{align*}
&\delta^{n}_H(c_H)(\mathrm{id} \otimes\psi_g^{\otimes n+1})([r];~a_1,\ldots,a_{n+1})\\
&=\delta^n_Hc_H([r];\psi_g(a_1),\ldots,\psi_g(a_{n+1}))\\
&=\delta^n_Hc_H([r];g a_1,\ldots,g a_{n+1})\\
      &= \theta_1 \big( R_0 (2;1,n) [r]; ~g a_1, c_H (R_2 (2;1,n)[r]; g a_2, \ldots, g a_{n+1})   \big)\\
      &~+ \sum_{i=1}^n  (-1)^i~ c_H \big(  R_0 (n; 1, \ldots, 2, \ldots, 1)[r];
       g a_1, \ldots, g a_{i-1}, \pi_A (R_i (1, \ldots, 2, \ldots, 1)[r]; g a_i, g a_{i+1}), g a_{i+2}, \ldots, g a_{n+1}   \big)\\
      &~+(-1)^{n+1} ~\theta_2 \big( R_0 (2; n, 1) [r];~ c_H (R_1 (2;n,1)[r]; g a_1, \ldots, g a_n), g a_{n+1}   \big)\\
    &= \theta_1 \big( R_0 (2;1,n) [r]; ~g a_1, c_H\circ(\mathrm{id} \otimes\psi_g^{\otimes n}) (R_2 (2;1,n)[r]; a_2, \ldots, a_{n+1})   \big)\\
      &~+ \sum_{i=1}^n  (-1)^i~ c_H \big(  R_0 (n; 1, \ldots, 2, \ldots, 1)[r];
       g a_1, \ldots, g a_{i-1}, g \pi_A (R_i (1, \ldots, 2, \ldots, 1)[r]; a_i, a_{i+1}), g a_{i+2}, \ldots, g a_{n+1}   \big)\\
      &~+(-1)^{n+1} ~\theta_2 \big( R_0 (2; n, 1) [r];~ c_H\circ(\mathrm{id} \otimes\psi_g^{\otimes n}) (R_1 (2;n,1)[r];  a_1, \ldots, a_n), g a_{n+1}   \big)\\
      &= \theta_1 \big( R_0 (2;1,n) [r]; ~g a_1, (\psi_g\circ c_K) (R_2 (2;1,n)[r]; a_2, \ldots, a_{n+1})   \big)\\
      &~+ \sum_{i=1}^n  (-1)^i~ c_H\circ(\mathrm{id} \otimes\psi_g^{\otimes n}) \big(  R_0 (n; 1, \ldots, 2, \ldots, 1)[r];
        a_1, \ldots,  a_{i-1},  \pi_A (R_i (1, \ldots, 2, \ldots, 1)[r]; a_i, a_{i+1}),  a_{i+2}, \ldots,  a_{n+1}   \big)\\
      &~+(-1)^{n+1} ~ g \theta_2 \big( R_0 (2; n, 1) [r];~ (\psi_g\circ c_K) (R_1 (2;n,1)[r];  a_1, \ldots, a_n), a_{n+1}   \big)\\
            &= \psi_g \circ \theta_1 \big( R_0 (2;1,n) [r]; ~ a_1, c_K(R_2 (2;1,n)[r]; a_2, \ldots, a_{n+1})   \big)\\
      &~+ \sum_{i=1}^n  (-1)^i~ (\psi_g\circ c_K) \big(  R_0 (n; 1, \ldots, 2, \ldots, 1)[r];
        a_1, \ldots,  a_{i-1},  \pi_A (R_i (1, \ldots, 2, \ldots, 1)[r]; a_i, a_{i+1}),  a_{i+2}, \ldots,  a_{n+1}   \big)\\
      &~+(-1)^{n+1} ~ \psi_g \circ \theta_2 \big( R_0 (2; n, 1) [r];~ c_K (R_1 (2;n,1)[r];  a_1, \ldots, a_n), a_{n+1}   \big)\\
       %     &= \psi_g \big(\theta_1 \big( R_0 (2;1,n) [r]; ~ a_1, c_K(R_2 (2;1,n)[r]; a_2, \ldots, a_{n+1})   \big)\\
%      &~+ \sum_{i=1}^n  (-1)^i~c_K \big(  R_0 (n; 1, \ldots, 2, \ldots, 1)[r];   a_1, \ldots,  a_{i-1},  \pi_A (R_i (1, \ldots, 2, \ldots, 1)[r]; a_i, a_{i+1}),  a_{i+2}, \ldots,  a_{n+1}   \big)\\
%      &~+(-1)^{n+1} ~ \theta_2 \big( R_0 (2; n, 1) [r];~ c_K (R_1 (2;n,1)[r];  a_1, \ldots, a_n), a_{n+1}   \big)\big)\\
     &= \psi_g\circ \delta^n_Kc_K([r];~a_1,a_2,\ldots,a_{n+1}).
\end{align*}
Therefore, $\delta^n_H(c_H)\circ (\mathrm{id} \otimes\psi_g^{\otimes {n+1}})=\psi_g\circ \delta^n_Kc_K$ and we have
%\begin{align*}
$\lbrace\delta_H(c_H)\rbrace \in S_G^{n+1} (\Phi_A;\Phi_A).$
%\end{align*}
\end{proof}

Thus, we obtain a cochain subcomplex
$\lbrace S_G^\sharp(\Phi_A;\Phi_A),\delta\rbrace.$
The homology groups $H_*(S^\sharp_G(\Phi_A;\Phi_A))$ of this cochain complex is called the cohomology of the $O_G$-dendriform algebra $\Phi_A$ associated with $A$ and coefficients in $\Phi_A$.

%\section{Equivariant abelian extension}
\begin{thm}
Let $A$ be a dendriform algebra with a given action of $G$. For the $O_G$-dendriform algebra $\Phi_A$, we have an isomorphism
$$H^n_{G}(A, A) \cong H_n(S^\sharp (\Phi_A, \Phi_A)).$$
\end{thm}
\begin{proof}
Let $\alpha \in C^n_{G} (A, A)= \text{Hom}_{\mathrm{Rep}(G)} (\mathbb{K}[C_n] \otimes A^{\otimes n}, A )$, where the representation of $G$ on $\mathbb{K}[C_n] \otimes A^{\otimes n}$ is given by $\mathrm{id}_{\mathbb{K}[C_n]} \otimes \varrho_{A^{\otimes n}}$. 
%In this proof we take $\varrho_{V^{\otimes n}}$ as $\psi^{\otimes n}_g$. 
Let us denote the restriction of $\alpha$ to $\mathbb{K}[C_n] \otimes (A^H)^{\otimes n}$ by $\alpha_H$. Note that $\lbrace \alpha_H \rbrace \in S^n_G (\Phi_A, \Phi_A)$.  For all $n \geq 1$, we define a map
\begin{align*}
F_n: C^n_{G} (A, A) \rightarrow S^n_G (\Phi_A, \Phi_A),~~~\alpha \mapsto \lbrace \alpha_H \rbrace.
\end{align*}
We claim that $F= \lbrace F_n \rbrace$ is a cochain map. Note that $(\delta \circ F_n)(\alpha)= \lbrace \delta_H (\alpha_H) \rbrace$. On the other hand, $(F_n \circ \delta)(\alpha) = F_n (\delta (\alpha)) = \lbrace \delta (\alpha)_H \rbrace = \lbrace \delta_H (\alpha_H) \rbrace$.
Thus, $F= \lbrace F_n \rbrace$ is a cochain map. Hence, $F$ induces a group homomorphism
\begin{align}
\bar {F_n} : H^n_{G}(A, A) \rightarrow H_n(S^\sharp (\Phi_A, \Phi_A)), \text{ for all }n \geq 1.
\end{align}

Now we define an inverse map $G = \lbrace G_n \rbrace$ as follows:
$$G_n : S^n_G (\Phi_A, \Phi_A) \rightarrow C^n_{G} (A, A),~~~\lbrace \alpha_H \rbrace \mapsto \alpha_{\lbrace e \rbrace}.$$
It is easy to verify that $G$ is an inverse of $F$. Therefore, $\bar {F_n}$ is an isomorphism for all $n$.
\end{proof}
\section{Equivariant one-parameter deformations of dendriform algebras}
In this section, we study formal one-parameter deformations of equivariant dendriform algebras. We show that such deformations are governed by the cohomology introduced in the previous section.

 \begin{defn} \label{deformation}A formal one-parameter deformation of $\Phi_A$ consists of a pair of natural transformations $(\mu^l_t, \mu^r_t)$ such that for every $G/H \in \text{Ob}(O_G)$, components of $\mu^l_t$ and $\mu^r_t$ are bilinear maps
 $$\mu^{l}_t (G/H), ~\mu^{r}_t(G/H) : \Phi_A(G/H)[[t]]\times \Phi_A(G/H)[[t]]\to \Phi_A(G/H)[[t]].$$
 This is same as
 \begin{align*}
 \mu^{H, l}_t, \mu^{H, r}_t : A^H[[t]]\times A^H[[t]]\to A^H[[t]],
 \end{align*}
 which can be expressed in the following form
 \begin{align*}
 &\mu^{H,l}_t(a,b)= \mu^{H,l}_0(a,b)+ \mu^{H,l}_1(a,b)t+ \mu^{H,l}_2(a,b)t^2+\cdots,\\
 &\mu^{H,r}_t(a,b)= \mu^{H,r}_0(a,b)+\mu^{H,r}_1(a,b)t+ \mu^{H,r}_2(a,b)t^2+\cdots,
 \end{align*}
 such that
 \begin{enumerate}
 \item $\mu^{H,l}_0(a,b)=a\prec b$ and $\mu^{H,r}_0(a,b)=a\succ b$, for $a,b\in A^H$, is the original multiplication of $A^H$.
\item For $i\geq 0$, $\mu^{H,l}_i$ and $\mu^{H,r}_i$ are  $\mathbb{K}$-bilinear maps $A^H\times A^H\to A^H$ and $\mu^{H,l}_t$ and $\mu^{H,r}_t$ satisfies the dendriform algebra relations. \label{dendalgebra relation in defor}
 \item For every morphisms $\hat{g}:G/H\to G/K$ and $g\in G$ satisfying $g^{-1}Hg\subseteq K$, the following relations hold \label{action compat in defor}
 \begin{align*}
& \mu^{H,l}_t\circ (\Phi_A(\hat{g})\otimes \Phi_A(\hat{g}))=\Phi_A(\hat{g})\circ \mu^{K,l}_t,\\
& \mu^{H,r}_t\circ (\Phi_A(\hat{g})\otimes \Phi_A(\hat{g}))=\Phi_A(\hat{g})\circ \mu^{K,r}_t.
  \end{align*}
 \end{enumerate}
 \end{defn}
 
Note that the condition (\ref{dendalgebra relation in defor}) of Definition \ref{deformation} gives rise to following identities
\begin{align}
 \mu^{H,l}_t \big(\mu^{H,l}_t (a, b ),  c\big) =~&  \mu^{H,l}_t \big (a,  (\mu^{H,l}_t (b, c) +  \mu^{H,r}_t(b, c) \big), \label{def-eqnn1}\\
 \mu^{H,l}_t \big(\mu^{H,r}_t (a, b),  c \big) =~&  \mu^{H,r}_t \big(a, \mu^{H,l}_t (b, c)\big), \label{def-eqnn2}\\
 \mu^{H,r}_t \big(\mu^{H,l}_t(a, b) + \mu^{H,r}_t (a, b), c \big)=~& \mu^{H,r}_t \big(a, \mu^{H,r}_t (b, c) \big), \label{def-eqnn3}
\end{align}
for all $a, b, c \in A^H$. This is equivalent to following system of equations: for all $n \geq 0$,
\begin{align}
& \label{dendalgebra relation ex1 in defor} \sum_{i+j=n} \mu^{H,l}_i \big(\mu^{H,l}_j (a, b ),  c\big) -  \mu^{H,l}_i \big (a,  (\mu^{H,l}_j (b, c) +  \mu^{H,r}_j(b, c) \big)= 0,\\
&\label{dendalgebra relation ex2 in defor}\sum_{i+j=n} \mu^{H,l}_i \big(\mu^{H,r}_j (a, b),  c \big) - \mu^{H,r}_ i \big(a, \mu^{H,l}_j (b, c)\big)=0,\\
&\label{dendalgebra relation ex3 in defor}\sum_{i+j=n} \mu^{H,r}_i \big(\mu^{H,l}_j(a, b) + \mu^{H,r}_j (a, b), c \big) - \mu^{H,r}_i \big(a, \mu^{H,r}_j (b, c) \big)=0,
\end{align}
for all $a, b, c \in A^H.$

\begin{defn}
For each $i \geq 0$ and $H\leq G$, define a map $\pi^H_i : \mathbb{K}[C_2] \otimes {(A^H)}^{\otimes 2} \rightarrow A^H$ by
$$\pi^H_i ([r]; a, b) = \begin{cases} \mu^{H,l}_i (a, b) ~~~ &\mathrm{ if } ~~ [r]=[1] \\
\mu^{H,r}_i (a, b) ~~~ &\mathrm{ if } ~~ [r] = [2]. \end{cases}$$
%For $n=1$, we have a notion of an infinitesimal of the deformation.
For $i=0$, we have $\pi^H_0 = \pi_{A^H}$.
The infinitesimal of the above equivariant one-parameter formal deformation is defined as
$$\pi_1=\bigoplus_{H\leq G} \pi^H_1.$$
\end{defn}

  \begin{remark}
 Using the condition (\ref{action compat in defor}) of Definition \ref{deformation}, it is easy to see that for each $i \geq 0$ and $H\leq G$, the map $\pi^H_i$ respect the group action, that is, $\pi^H_i ([r]; ga, gb)=g \pi^H_i ([r]; a, b)$. 
 \end{remark}
 
 \begin{prop}
 The infinitesimal of an equivariant deformation of a dendriform algebra is a $2$-cocycle.
 \end{prop}
 \begin{proof}
 Using the circle product $\circ$ as defined in (\ref{dend-operad-circ}), we can rewrite the equations (\ref{dendalgebra relation ex1 in defor}), (\ref{dendalgebra relation ex2 in defor}), (\ref{dendalgebra relation ex3 in defor}) in a compact form as follows:
 $$\sum_{i+j=n} \pi^H_i \circ \pi^H_j= 0,~~~\forall~~~ n\geq 0,~~~\forall~~~H\leq G.$$
 For $n=1$, we get
 $$\pi_{A^H} \circ \pi^H_1 + \pi^H_1 \circ \pi_{A^H}=0.$$
 This implies, $\delta^2_H \pi^H_1=0$. Therefore, $\pi_1=\bigoplus_{H\leq G} \pi^H_1$ is an equivariant $2$-cocycle.
 \end{proof}
 
  \subsection{Rigidity of equivariant deformation}
  The aim of this subsection is to discuss the notion of equivalence of equivariant deformations of dendriform algebras.
 \begin{defn}
 \label{equ equivalent defn} Let $( \mu^l_t, \mu^r_t)$ and $(\mu\rq^l_t, \mu\rq^r_t)$ be two equivariant deformations of  the dendriform algebra $A$ equipped with an action of $G$, where $\mu^k_t=\sum_{i\geq 0} \mu^k_i t^i$ and $\mu\rq^k_t=\sum_{i\geq 0} \mu\rq^k_it^i$ for $k= l, r$.  We say that $( \mu^l_t, \mu^r_t)$ and $(\mu\rq^l_t, \mu\rq^r_t)$ are equivalent if there is a formal equivariant isomorphism $\Psi_t: A[[t]]\to A[[t]]$ of the following form:
$$\Psi_t(a)=\psi_0(a)+\psi_1(a)t+\psi_2(a)t^2+\cdots$$
such that 
\begin{enumerate}
\item $\psi_0= \mathrm{id}$ and for $i\geq 1$, $\psi_i: A\to A$ is an equivariant $\mathbb{K}$-linear map;
\item $\Psi_t\circ \mu_t\rq^l=\mu^l_t\circ (\Psi_t\otimes \Psi_t)\,\,\,\text{and}\,\,\, \Psi_t\circ \mu_t\rq^r=\mu^r_t\circ (\Psi_t\otimes \Psi_t).$
\end{enumerate}
\end{defn}
\begin{remark}
Suppose  $( \mu^l_t, \mu^r_t)$ and $(\mu\rq^l_t, \mu\rq^r_t)$ are equivalent deformations via equivariant isomorphism $\Psi_t$. Then for every subgroup $H\leq G$, the equivariant isomorphism $\Psi_t$ induces a formal isomorphism $A^H[[t]]\to {A}^H[[t]]$ for all subgroups $H$ of $G$.
\end{remark}
From the second condition of the Definition \ref{equ equivalent defn} we have the following equations
 \begin{align}
 &\sum_{i,j\geq 0}\psi_i(\mu\rq^l_j(a,b))t^{i+j}=\sum_{i,j,k\geq 0}\mu^l_i(\psi_j(a),\psi_k(b))t^{i+j+k},\\
 &\sum_{i,j\geq 0}\psi_i(\mu\rq^r_j(a,b))t^{i+j}=\sum_{i,j,k\geq 0}\mu^r_i(\psi_j(a),\psi_k(b))t^{i+j+k}.
 \end{align}
 
 Then we have the following.
 \begin{prop}
The infinitesimals corresponding to equivalent equivariant deformations are cohomologous. Thus, they give rise to the same cohomology class.
  \end{prop}
  
  \begin{proof}
  The conditions of the above definition are equivalent to
  $$\sum_{i+j=n} \psi_i ( \pi_j ([r]; a, b)) = \sum_{i+j+k=n} \pi\rq_k ([r]; \psi_i(a), \psi_j(b)).$$
  For $n=1$, we have,
  \begin{align*}
  \pi_1([r]; a, b) + \psi_1 ( \pi_A ([r]; a, b))
  = \pi_A ([r]; \psi(a), b) + \pi_A ([r]; a, \psi(b)) + \pi\rq_1 ([r]; a, b).
  \end{align*}
  In other words, 
  \begin{align*}
  (\pi- \pi\rq)([r]; a, b) &= \pi_A ([r]; \psi(a), b) + \pi_A ([r]; a, \psi(b)) -  \psi_1 \circ \pi_A ([r]; a, b)\\
                                  &= (\pi_A \circ \psi_1 - \psi_1 \circ \pi_A)([r]; a, b) = \delta^2 (\psi_1).
   \end{align*}
   Note that $\phi_1$ is an equivariant map and the above equality shows that the infinitesimals of two equivalent equivariant deformations are cohomologous.
  \end{proof}
  \begin{defn}
  A dendriform algebra $A$ equipped with an action of $G$ is called rigid if every equivariant deformation of $A$ is equivalent to the trivial deformation.
  \end{defn}
  Similar to the non-equivariant case \cite[Theorem 3.5]{Das}, we have the following rigidity theorem for equivariant deformations.
 \begin{thm}
 Let $A$ be a dendriform algebra equipped with an action of finite group $G$. If $H^2_{G}(A, A)=0$ then $A$ is equivariantly rigid.
 \end{thm}
 
 \subsection{Extensions of finite order deformations}
In this final subsection, we discuss the problem of extending an equivariant dendriform algebra deformation of $A$ of order $n$ to that of order $n+1$. 
\begin{defn}\label{n-deform}
An equivariant $n$-deformation of a dendriform algebra is given by sums $\mu^l_t = \bigoplus_{H \leq G} \mu^{H, l}_t$ and $\mu^r_t = \bigoplus_{H \leq G} \mu^{H, r}_t$, where
$$\mu^{H, l}_t=\sum^n_{i=0} \mu^{H, l}_i t^i \quad \mathrm{and} \quad \mu^{H, r}_t=\sum^n_{i=0} \mu^{H, r}_i t^i$$
such that (\ref{def-eqnn1})-(\ref{def-eqnn3}) holds modulo $t^{n+1}.$
\end{defn}

We say an equivariant $n$-deformation $(\mu^l_t, \mu^r_t)$ of a dendriform algebra is extendible to an equivariant $(n+1)$-deformation if for all subgroups $H$ of $G$ there is an element $\mu^{H, l}_{n+1}, \mu^{H, r}_{n+1}\in C^{2}_{\mathrm{dend}}(A^H, A^H)$ such that
\begin{align*}
 \bar{\mu_t}^{H, l}= \mu^{H, l}_t + \mu^{H, l}_{n+1}t^{n+1}  \quad \mathrm{and} \quad
 \bar{\mu_t}^{H, r}= \mu^{H, r}_t + \mu^{H, r}_{n+1}t^{n+1},
\end{align*}
defines a $(n+1)$-deformation of $A^H$.

 Since $(\mu^{H, l}_t, \mu^{H, r}_t)$ is a deformation of order $n$, we have
 $$\pi_{A^H} \circ \pi^H_i + \pi^H_1 \circ \pi^H_{i-1} + \cdots + \pi^H_{i-1} \circ \pi^H_1 + \pi^H_i \circ \pi_{A^H}= 0,~~~\text{for}~~~ i=1,\ldots, n.$$
In other words, 
 $$\delta^H_{\mathrm{dend}} (\pi^H_i) = - \sum _{\substack{p+q=i\\p, q \geq 1}} \pi^H_p \circ \pi^H_q .$$
We define
\begin{align*}
\pi^H_{n+1} ([r]; a, b)= \begin{cases}
                                   & \mu^{H, l}_{n+1} ~~~~ \text{if}~~~ [r]=[1],\\
                                   & \mu^{H, r}_{n+1} ~~~~ \text{if}~~~ [r]=[2].
                                  \end{cases}
\end{align*} 
 
 For $(\bar{\mu_t}^{H, l}, \bar{\mu_t}^{H, r})$ to be a deformation of order $n+1$, one more deformation equation need to be satisfied, namely,
$$\delta^H_{\mathrm{dend}} (\pi^H_{n+1}) = - \sum _{\substack{p+q=n+1\\ p, q \geq 1}} \pi^H_p \circ \pi^H_q .$$
We define
$$\mathrm{Ob}^H = - \sum _{\substack{p+q=n+1\\ p, q \geq 1}} \pi^H_p \circ \pi^H_q.$$
The equivariant obstruction to extend the given deformation is defined as $$\mathrm{Ob}^G = \bigoplus_{H \leq G} \mathrm{Ob}^H.$$
 \begin{thm}
The equivariant obstruction is a $3$-cocycle in the equivariant cohomology of $A$.
 \end{thm}
 \begin{proof}
 First we show that $\mathrm{Ob}^G \in S^3 (\Phi_A, \Phi_A)$ is invariant under the action of $G$.
 As $\pi_i$ respect the group action, it is easy to see that $\mathrm{Ob}^H$ is an invariant $3$-cocycle.
 Thus, $\mathrm{Ob}^G \in S_G^3 (\Phi_A, \Phi_A)$. 
 
 Note that $\mathrm{Ob}^G \in S_G^3 (\Phi_A, \Phi_A) \leq S^3 (\Phi_A, \Phi_A)$ is also a obstruction cocycle for the non-equivariant extension of the given deformation and $\delta :  S_G^3 (\Phi_A, \Phi_A) \rightarrow  S_G^4 (\Phi_A, \Phi_A)$ is the restriction of the non-equivariant coboundary map $\delta :  S^3 (\Phi_A, \Phi_A) \rightarrow  S^4 (\Phi_A, \Phi_A)$ to the submodule  $S_G^3 (\Phi_A, \Phi_A)$.  Therefore, $\mathrm{Ob}^G$ is a $3$-cocycle in the equivariant cohomology of the dendriform algebra.
 \end{proof}
 
 As a consequence, we have the following.
 \begin{thm}
 If $H^3_G(A;A)=0$ then every finite order equivariant deformation of $A$ extends to a deformation of next order.
 \end{thm}
 
\begin{corollary}
If $H^3_G(A;A)=0$ then every $2$-cocycle is the infinitesimal of some equivariant formal deformation of $A$. 
 \end{corollary}

 \vspace*{.2cm} 
 
\noindent {\bf Acknowledgements.} The research of A. Das is supported by the fellowship of Indian Institute of Technology (IIT) Kanpur.

\mbox{ }\\

\providecommand{\bysame}{\leavevmode\hbox to3em{\hrulefill}\thinspace}
\providecommand{\MR}{\relax\ifhmode\unskip\space\fi MR }
% \MRhref is called by the amsart/book/proc definition of \MR.
\providecommand{\MRhref}[2]{%
  \href{http://www.ams.org/mathscinet-getitem?mr=#1}{#2}
}
\providecommand{\href}[2]{#2}

\mbox{ } \\
\end{document}